\documentclass[10pt,a4paper]{amsart}
\usepackage{latexsym,amssymb,amsmath,amsthm,amsfonts,enumerate,verbatim,xspace,
exscale}

\usepackage{color,amsmath,amsbsy,amsfonts,amssymb}
 
\usepackage{hyperref}
\definecolor{darkblue}{rgb}{0,0,.5}
\hypersetup{pdftex=true, colorlinks=true, breaklinks=true, linkcolor=darkblue, menucolor=darkblue, pagecolor=darkblue, urlcolor=darkblue,citecolor=darkblue}

\theoremstyle{plain}
\newtheorem{theorem}{Theorem}[section]
\newtheorem{lemma}[theorem]{Lemma}
\newtheorem{proposition}[theorem]{Proposition}

\theoremstyle{definition}

\def\d{\textup{div}}
\def\D{\mathcal{D}}
\def\R{\mathbb{R}}

\def\A{\mathcal{A}}
\def\Ham{\mathcal{H}}

\newcommand{\up}{\upshape}

\newcommand{\leftmapsto}{\mbox{$\;\leftarrow\!\mapstochar\;$}}

\newcommand{\longto}{\longrightarrow}
\newcommand{\hookto}{\hookrightarrow}
\newcommand{\toto}{\twoheadrightarrow}

\def\vv<#1>{\langle#1\rangle}

\newcommand{\diag}[1]{\mbox{$\textup{diag}(#1)$}}

\newcommand{\pr}{\mbox{$\text{\up{pr}}$}}

\providecommand{\det}{\mbox{$\text{\up{det}}\,$}}

\providecommand{\vol}{\mbox{$\text{\up{vol}}$}}

\newcommand{\dd}[2]{\mbox{$\frac{\partial #2}{\partial #1}$}}

\newcommand{\Om}{\Omega}
\newcommand{\var}{\varphi}

\newcommand{\lam}{\lambda}
\newcommand{\Lam}{\Lambda}

\newcommand{\wt}[1]{\mbox{$\widetilde{#1}$}}

\newcommand{\by}[2]{\mbox{$\frac{#1}{#2}$}}

\providecommand{\set}[1]{\mbox{$\{#1\}$}}

\newcommand{\X}{\mathfrak{X}}

\newcommand{\curv}{\mbox{$\textup{Curv}$}}

\newcommand{\ver}{\mbox{$\textup{Ver}$}}
\newcommand{\hor}{\mbox{$\textup{Hor}$}}

\newcommand{\momap}{momentum map\xspace}

\newcommand{\ao}{\mathfrak{a}}

\newcommand{\gu}{\mathfrak{g}}
\newcommand{\ho}{\mathfrak{h}}

\newcommand{\ko}{\mathfrak{k}}

\newcommand{\mo}{\mathfrak{m}}

\newcommand{\po}{\mathfrak{p}}

\newcommand{\Ad}{\mbox{$\text{\upshape{Ad}}$}}
\newcommand{\ad}{\mbox{$\text{\upshape{ad}}$}}

\newcommand{\orb}{\mbox{$\mathcal{O}$}}

\newcommand{\SL}{\mbox{$\textup{SL}$}}

\newcommand{\gl}{\mbox{$\mathfrak{gl}$}}

\newcommand{\SO}{\mbox{$\textup{SO}$}}
\newcommand{\so}{\mbox{$\mathfrak{so}$}}

\newcommand{\Hamc}{\mbox{$\mathcal{H}_{\textup{c}}$}}
\newcommand{\Omnh}{\mbox{$\Om_{\textup{nh}}$}}
\newcommand{\Xnh}{\mbox{$X_{\textup{nh}}$}}
\newcommand{\W}{\mbox{$\mathcal{W}$}}

\title[Chaplygin systems associated to semi-simple Lie
  groups]{Chaplygin systems associated to Cartan decompositions of semi-simple Lie groups}
 
\author{Simon Hochgerner}
\address{Section de Mathematiques,
Station 8,
EPFL, CH-1015 Lausanne}
\email{simon.hochgerner@epfl.ch} 
\date{July 1, 2009}

\dedicatory{Dedicated to Peter Michor on the Occasion of his 60th Birthday}

\begin{document}

\begin{abstract}
We relate a Chaplygin type system to a Cartan decomposition of a real
semi-simple Lie group. The resulting system is described in terms
of the structure theory associated to the Cartan decomposition. It
is shown to possess a preserved measure and when internal symmetries
are present these are factored out via a process called
truncation. Furthermore, a criterion for Hamiltonizability of the
system on the so-called ultimate reduced level is given. As important
special cases we find the Chaplygin ball rolling on a table and the
rubber ball rolling over another ball. 
\end{abstract}

\maketitle
\tableofcontents

\section{Introduction} 

We generalize the $n$-dimensional Chaplygin ball problem 
\cite{Chap87,FK95,EKMR04,D04,Jov09,HG09} to
non-holonomic systems associated to 
semi-simple Lie groups, 
and show how the Chaplygin ball system arises as a special case.
That is, we consider a real semi-simple Lie group $G$ and a Cartan
decomposition $G\cong K\times\po$ in the common notation of \cite{K02}.
On the Lie algebra level we have 
$\gu=\ko\oplus\po$ together with the usual bracket relations. In
$\po$ we fix a maximal abelian subspace $\ao$ and an element
$w_0\in\ao$. In Section~\ref{sec:chap-based-on-lie} we define a
non-holonomic system that is naturally associated to these data: 
the configuration space is 
\[
 Q := K\times V
\]
where $V$ is orthogonal to 
$Z_{\mathfrak{p}}(w_0) = \set{x\in\po: [w_0,x]=0}$ 
within $\po$, the constraint distribution is 
\[
 \D 
 := 
 \set{(s,u,x,[w_0,\Ad(s)u])\in K\times\ko\times V\times V}
 \subset TQ,
\]
and the Lagrangian is the obvious left invariant kinetic energy function
on $TQ$.
Then we use
the restricted roots of the pair $(\gu,\ao)$ to give a detailed
description
of the this model. We will see that the $n$-D Chaplygin ball
corresponds to taking $G=\SO(n,1)$.

We extend some of the results of \cite{FK95,Jov09,HG09}
to this setting. In particular this yields a geometrization of these
results since we follow the philosophy of \cite{EKMR04} in working with
a global trivialization of the compressed phase space
and using (almost) symplectic techniques.

More precisely, by making use of the restricted root space
decomposition associated to $(\gu,\ao)$ we directly show the existence of a preserved measure
for these types of systems at the compressed level -- Proposition~\ref{thm:p-meas}.

Then we pass to the
ultimate reduced phase space by means of truncation and reduction of
internal symmetries. This involves changing the non-holonomic two-form
in a certain way that is better adapted to the symmetries -- Section~\ref{sec:trunc}.
The passage from the original non-holonomic system to this reduced
phase space via compression
followed by reduction of internal symmetries is reminiscent of the
Hamiltonian reduction in stages theory which also lends the
terminology `ultimate reduced space'.

Moreover, 
in Theorem~\ref{thm:ham-at-0} we derive a necessary and sufficient condition for
Hamiltonization of the ultimate reduced system when the angular momentum with
respect to the internal symmetries is fixed to $0$.
This condition is of algebraic nature
and in some simple cases it allows to decide (non-) Hamiltonizability
by looking at the root system of $(\gu,\ao)$.
This result is a statement which only holds at the ultimate reduced
level and thus depends crucially on the reduction by truncation
described in 
Section~\ref{sec:trunc}. 


Section~\ref{sec:examples} contains some examples. We return to the
$n$-dimensional Chaplygin ball system corresponding to $G=\SO(n,1)$ 
and apply Theorem~\ref{thm:ham-at-0} to verify the recent result of
Jovanovic~\cite{Jov09} on Hamiltonizability of this system at the
ultimate reduced level when the angular momentum is fixed to $0$ and 
the inertia tensor is of special
form. Then we give two examples related to $\SL(n,\R)$ and
$\textup{Sp}(n,\R)$. 

Finally, we show how the rubber rolling
sphere-on-sphere system arises in this setting. This is not so
straightforward as for the ball on a table: We start with the split
real form of the complex semi-simple Lie group $G_2$ and consider,
according to the recipe of Section~\ref{sec:chap-based-on-lie}, its
Cartan decomposition. The resulting system is shown to be never
Hamiltonizable, not even for homogeneous inertia tensor $\mathbb{I}=1$. 
However, from Koiller and Ehlers~\cite{EK07} we know
that the rubber rolling system is Hamiltonizable. Thus we are
motivated to find a subsystem which is an obvious candidate for allowing
Hamiltonizability. This subsystem is then recognized as the rubber
ball arrangement for the case in which the ratio of the radii of the
balls is $1:3$. However, we are not claiming that we provide any new
insights into the dynamics of this system; we only find a new way to
see this as being part of a non-holonomic system that is
naturally defined on some bigger phase space.

In Section~\ref{sec:roh}
we recall the notion of Hamiltonization of a non-holonomic system.
Then we
reformulate the Chaplygin multiplier theorem in terms of a
characterization of conformally closed almost symplectic forms which
is due to Libermann~\cite{L55,LM87}. This characterization extends to
higher dimensions whence we also formulate a higher dimensional
analogon of the multiplier theorem. In Section~\ref{sec:ham} this is
used as a preparation for Theorem~\ref{thm:ham-at-0}.

\textbf{Acknowledgements.}
I would like to thank the organizers of the Mikulov meeting, Special
edition in honor of Peter Michor's 60th birthday, for two pleasant
days and the opportunity to present parts of the present paper. The
idea of relating Chaplygin systems to semi-simple Lie groups is, of
course,  taken
from the paper \cite{AKLM03} of Peter et al.\ where a similar
programme is carried out for spin Calogero-Moser systems. 
I am
also grateful to Tudor Ratiu for helpful discussions and to Katja
Sagerschnig for important remarks concerning
Section~\ref{sec:rubber}.

\section{Remarks on Hamiltonization}\label{sec:roh}

Non-holonomic systems can be seen as a generalization of Hamiltonian
mechanics. A natural question that arises is: when is a non-holonomic
system Hamiltonian or \emph{Hamiltonizable}?

As a toy example to illustrate some key ideas and also to set up notation 
we consider the vertical rolling disk. For more information on this,
and also on more complicated examples, see Bloch~\cite{B03}.
The configuration space is 
\[ 
 Q = S^1\times S^1\times\R^2
\]
with coordinates $q = (\theta,\var,x,y)$. Here $(x,y)$ denotes the contact
point of the disk on the table, $\theta$ its internal orientation, and
$\phi$ its orientation with respect to a fixed axis on the table. The
Lagrangian is the kinetic energy
\[
 L
 =
 \by{1}{2}\mathbb{I}\dot{\theta}^2 
 + \by{1}{2}\mathbb{J}\dot{\var}^2
 + \by{1}{2}m(\dot{x}^2+\dot{y}^2) 
\]
where $m$ is the mass of the disk and $\mathbb{I}$ and $\mathbb{J}$
are the different moments of inertia of the disk. The motion is to
satisfy a no slip constraint which means that 
\[
 \dot{x} = R\dot{\theta}\cos\var
 \textup{ and }
 \dot{y} = R\dot{\theta}\sin\var
\]
where $R$ is the radius of the disk.
To rewrite these constraints in a more geometric manner consider the
$\R^2$-valued $1$-form $\A\in\Om^1(S,\R^2)$ on $S := S^1\times S^1$ given by
\[
 \A_{(\theta,\var)} 
 =  
 \left(
 \begin{matrix}
  -R\cos\var\, d\theta\\
  -R\sin\var\, d\theta
 \end{matrix}
 \right).
\]
Let $\pi: Q = S\times\R^2\to S$ denote the Cartesian projection. 
The constraint space is thus defined by the smooth distribution 
\[ 
 \D =
 \set{(q,\dot{\theta},\dot{\var},-\A_{\pi(q)}(\dot{\theta},\dot{\var}))}
 \subset TQ
\]
Now it is important to notice that $L$ and $\D$ are invariant under
the free and proper action of the abelian Lie group $\R^2$ on $TQ$.  
This action defines a (trivial) principal fiber bundle $\R^2\hookto Q\toto
S$. Moreover, $\D$ is complementary to the vertical space $\ker T\pi$
of this bundle. In other words $\D$ defines a principal connection
with connection form $\A$ and the non-holonomic system $(Q,L,\D)$ is a
\emph{$G$-Chaplygin system} with $G=\R^2$. 
This system is truly non-holonomic since $\D$ is non-integrable since
the curvature $\curv_0^{\mathcal{A}} = d\A$ is non-zero.

$G$-Chaplygin systems are very well
behaved in the sense that they allow for a natural reduction of symmetries. For
this our main reference is \cite{EKMR04} where this reduction is termed
\emph{compression}. See also \cite{BS93} for a more general reduction
and \cite{HG09} for an account of these facts in the present
notation. 
The compressed system turns out to be an almost Hamiltonian system on
$T^*S$ with compressed Hamiltonian $\Hamc$. Of course, $\Hamc$ is
obtained by taking the Legendre transform of $L$, restricting to the
appropriate constraint subspace and factoring out the symmetries. The
dynamics $\Xnh = (\Omnh)^{-1}d\Hamc$ 
of the compressed system are encoded in the almost symplectic
form 
\[
 \Omnh
 :=
 \Om^S -
 \vv<J\circ\textup{horLift}^{\mathcal{A}},\curv_0^{\mathcal{A}}>
 = 
 \Om^s + \vv<\A,d\A>
\]
where $\Om^S$ is the canonical symplectic form on $T^*S=TS$
(identified via induced Legendre transform),
$\textup{horLift}^{\mathcal{A}}: TS\to TQ$ 
is the horizontal lift,  
$J: TQ=T^*Q\to\R^{2*}=\R^2$ (Legendre transform) is the standard
\momap
associated to the $\R^2$-action,
and
$\curv_0^{\mathcal{A}}$ is the induced curvature form on $S$
pulled-back to $TS$. Note that $\vv<\A,d\A>$ is a semi-basic
two-form on $TS$ which depends linearly on the fibers; 
the $\A$ in the left hand side of the pairing is
viewed as a function on $TS$. 
In general, the term
$\vv<J\circ\textup{horLift}^{\mathcal{A}},\curv_0^{\mathcal{A}}>$ is
non-closed thus preventing the system form being Hamiltonian. However,
in this special example we have
\[ 
 \vv<\A,d\A>_{(\theta,\var,\dot{\theta},\dot{\var})}
 =
 R^2\vv<
  \left(
  \begin{matrix}
    \dot{\theta}\cos\var\\
    \dot{\theta}\sin\var
  \end{matrix}
  \right)
 ,
   \left(
  \begin{matrix}
    -\sin\var \,d\var\wedge d\theta\\
    \cos\var \,d\var\wedge d\theta
  \end{matrix}
  \right)
 >
 =
 0.
\]
Thus the compressed system $(TS,\Om^S,\Hamc)$ is Hamiltonian even
though we started from a truly non-holonomic system $(Q,L,\D)$. Of
course, this fact is neither new nor surprising: the constraint forces
for this system are trivial.

More generally it may turn out that $\Omnh$ is \emph{conformally symplectic}
with respect to a positive function $F: S\to\R$, that is,
$d(F\Omnh)=0$. If this is the case we consider the rescaled
vectorfield $F^{-1}\Xnh$ which is now Hamiltonian with respect to
$F\Omnh$, and we say that the system $(T^*S,\Omnh,\Hamc)$ is
\emph{Hamiltonizable} 
or that $(Q,L,\D)$ is 
\emph{Hamiltonizable at the compressed level}. The idea is that one
reparametrizes the time $t=F^{-1}\tau$ in an $F$-dependent manner so that the system
is Hamiltonian in the new time $\tau$.

\subsection{Chaplygin's multiplier theorem via Libermann's criterion}
Let $(M,\sigma)$ be an almost symplectic manifold of dimension $2m$,
that is,
$\sigma$ is non-degenerate. Then we will make use of the codifferential
operator 
\[
 \delta: \Om^k(M)\longto\Om^{2m-k}(M)
\]
which is built out of $\sigma$ in the same way that the Hodge
codifferential is built out of a metric.
This operator is explained in the first chapter of the book of
Libermann and Marle \cite{LM87} and we use the same conventions.

\begin{theorem}[Chaplygin]
Let $B$ be a 2-dimensional Riemannian manifold.
Consider the natural kinetic energy Hamiltonian $\Ham: T^*B\to\R$
associated to the metric. Let $(T^*B,\sigma,\Ham)$ be an almost
Hamiltonian system such that:
\begin{enumerate}[\up (1)]
\item
$\sigma = \Om+\Lambda$ 
where $\Lambda$ is semi-basic with respect to $T^*B\to B$ and linear
in the fiber. That is, locally, $\Lambda = l(q,p)dq^1\wedge dq^2$ with
$l$ linear in $p$. Further, $\Om = \Om^B+\Xi$ with $\Xi$ magnetic,
that is, closed and basic. 
\item
There is a function $F: B\to\R_{>0}$ such that $L_X(F\sigma^2)=0$
where $X$ is the vector field associated to $\Ham$ via $\sigma$.
\end{enumerate}
Then
\[ 
 \delta\sigma = -d(\log F)
 \text{ and } 
 d(F\sigma)=0.
\]
\end{theorem}

\begin{proof}
The following formula can be found in \cite{LM87}:
\[
 d\sigma = \delta\sigma\wedge\sigma
\]
which holds since $\dim B = 2$, and thus
\begin{equation}\label{e:2d-condition}
 d(f\sigma)
 =
 (\delta\sigma+d(\log f))\wedge f\sigma
\end{equation}
for an arbitrary smooth function $f: T^*B\to\R$.
Therefore,
\[
 0
 =
 L_X(F\sigma^2)
 =
 2d(Fd\Ham\wedge\sigma)
 =
 2(dF+F\delta\sigma)\wedge d\Ham\wedge\sigma.
\]
Using the special structure of $\Lam$ we can show that $\delta\sigma$
is basic. (See Lemma~\ref{l:delta-sigma}.)
Therefore, since $\Ham$ is natural it follows that  
$dF+F\delta\sigma = 0$. Thus $d(F\sigma)=0$ by \eqref{e:2d-condition}.
\end{proof}

In particular, this proves Hamiltonization of the $3D$-Chaplygin
ball at the ultimate reduced level
 -- the $T^*S^2$-level which can be
attained after truncation.  
It is remarkable that this theorem as well as its crucial assumption
-the preserved measure- had already been found by
Chaplygin. Nevertheless, he could not apply these facts to conclude
Hamiltonizability of the problem. This is probably due to the fact
that it is not entirely straightforward to reduce all the relevant
structure in a coherent manner to the $T^*S^2$-level. See
\cite{HG09}. Indeed, it was Borisov and Mamaev~\cite{BM01,BM05} who
invented a proof of Hamiltonizability of this system.

\subsection{A multiplier theorem for higher dimensions}

Let $(M,\sigma)$ be a $2m$-dimensional almost symplectic manifold with
codifferential $\delta$. According to \cite{L55},
\cite[Proposition~I.16.5]{LM87}  there is a certain (\emph{effective})
$3$-form
$\psi$ such that 
\begin{equation}\label{E:dsigma}
 d\sigma = \psi + \by{1}{m-1}\delta\sigma\wedge\sigma.
\end{equation}
Moreover, $\sigma$ is \emph{locally} conformal symplectic if and only
if $\psi = 0$. 

Thus for an almost Hamiltonian system
$(T^*B=M,\sigma,\Ham)$ with dynamics given by $X=\sigma^{-1}d\Ham$
there are two obvious necessary conditions for a function $F:
B\to\R_{>0}$ to be a conformal factor ($d(F\sigma)=0$). 
Firstly, $\psi = 0$. Secondly, there is a preserved measure,
$L_X(F^{m-1}\sigma^n)=0$. 

The following statement
attempts to reverse the situation: 
When $\psi$ vanishes we know
that the structure is locally conformally symplectic; when there is
additionally a preserved measure then we can turn this local statement
to a global one. 

In fact, we will consider a slightly more general situation by
allowing the almost Hamiltonian system to have additional internal
degrees of freedom: Let $H\hookto S\toto B$ be a principal fiber
bundle which is at the same time a Riemannian submersion. That is,
$(S,\mu_S)$ and $(B,\mu_B)$ are Riemannian manifolds, $\mu_S$ is
$H$-invariant and the bundle projection map induces an isometry
$\hor(\mu_S) = \ver^{\bot} \to TB$. Let us denote the connection form corresponding
to $\hor(\mu_S)$ by $A: TS\to\ho$. This is the mechanical connection
on $(S,\mu_S)$ (and should not be confused with the $\A$ appearing in
Section~\ref{sec:chap-based-on-lie}).
We suppose that $T^*S$ is equipped with an almost symplectic form
$\wt{\Om} := \Om^S+\Lam$ where $\Lam$ is $H$-basic with respect to 
$T^*S\toto (T^*S)/H$, 
semi-basic with respect to
$T^*S\to S$ and linear in the fibers of $T^*S$. Thus $\wt{\Om}$ admits
a \momap $J_H: T^*S\to\ho^*$ which is the standard one, since $\Lam$
vanishes upon insertion of infinitesimal generators of the
$H$-action.

Further, assume that there is a right Hamiltonian $H$-space
$(F,\Om^F)$ with equivariant \momap $J_F: F\to\ho^*$. 

Then we consider the diagonal action of $H$ on $T^*S\times F$ where
the $H$-action on the second factor is inverted to give a left
action. This action admits a \momap which is given by $J := J_H -
J_F$.  
Notice that $(s,u,f)\in J^{-1}(0)$ if and only if $u =
u_0+A_s^*(J_F(f))$ with $u_0\in\hor^*_s$.
Thus we may pass to the reduced space 
\[
 J^{-1}(0)/H
 \cong
 T^*B\times_B(S\times_H F) =: \mathcal{W}
\]
where the isomorphism is defined in terms of the connection $A$. 
In particular, the reduced space $\mathcal{W}$ is a (symplectic) 
fiber bundle over $T^*B$ with fiber $F$.
By construction the form $\wt{\Om}+\Om^F$ is basic when restricted to
$J^{-1}(0)$ and passes to an almost symplectic
form on 
$T^*B\times_B(S\times_H F)$ 
which we shall denote by
$
 \sigma_A
$
to emphasize the $A$-dependence.
This
is, of course, the Weinstein construction rewritten for a semi-basic
perturbation of the standard symplectic form on $T^*S$. 
By the usual computation one sees that
\begin{equation}\label{e:sa}
 \sigma_A
 =
 \Om^B - \vv<J_F,\curv^A> + \Lam_0 + \Om^F
\end{equation}
where $\Om^B$ is the canonical symplectic form on $T^*B$, the second
term is magnetic and $\Lam_0$ is the non-closed semi-basic term
induced from $\Lam$. 

The situation which we have in mind is that of
\cite[Corollary~4.2]{HG09}.

\begin{theorem}\label{thm:mult}
Consider the natural kinetic energy Hamiltonian $\Ham: T^*S\to\R$
associated to the metric $\mu_S$
and let $\Ham: \W\to\R$ also denote the induced function.
Let $m = \by{1}{2}\dim\mathcal{W}$, $n=\dim B$ and $k=\by{1}{2}\dim
F$, whence $m=n+k$.
Assume that:
\begin{enumerate}[\up (1)]
\item
There is a function $F: B\to\R_{>0}$ such that $L_X(F^{m-1}\sigma_A^m)=0$
where $X$ is the vector field associated to $\Ham$ via $\sigma_A$.
($\sigma_A^m=(\Om^B)^n\wedge(\Om^F)^k$.)
\item
$\psi = 0$, or, equivalently $d\sigma_A = \by{1}{m-1}\delta\sigma_A\wedge\sigma_A$.
\end{enumerate}
Then
\[
 (m-1)d\log F = -\delta\sigma_A
 \textup{ and }
 d(F\sigma_A)=0,
\]
that is, the almost Hamiltonian system $(\W,\sigma_A,\Ham)$ with
dynamics given by $X=\sigma_A^{-1}d\Ham$ can be transformed to a Hamiltonian
system $(\W,F\sigma_A,\Ham)$ with rescaled dynamics $F^{-1}X$. 
\end{theorem}

%

\begin{proof}
According to \eqref{E:dsigma} we have
\begin{equation}\label{e:d(fsigma)}
 d(f\sigma_A)
 =
 \by{1}{m-1}(\delta\sigma_A+(m-1)d\log f)\wedge f\sigma_A + f\psi
\end{equation}
for all smooth functions $f: \mathcal{W}\to\R_{>0}$. 

We use local Darboux coordinates $q^a,p_a$ on $T^*B$.
Because of Lemma~\ref{l:delta-sigma} the one-form $\delta\sigma_A$ is
basic. Thus we have 
\[ 
 (m-1)d\log F + \delta\sigma_A = \sum\phi_a(q)dq^a
\]
in the local coordinates. Now,
\begin{align*} 
 0
 &= 
 di_X(F^{m-1}\sigma_A^m)
 =
 md(F^{m-1}d\Ham\wedge\sigma_A^{m-1})\\
 &=
 m((m-1)F^{m-2}dF\wedge d\Ham\wedge\sigma_A^{m-1}
   - F^{m-1}d\Ham\wedge\delta\sigma_A\wedge\sigma_A\wedge\sigma_A^{m-2})\\
 &=
 mF^{m-1}((m-1)d\log F + \delta\sigma_A)\wedge d\Ham\wedge\sigma_A^{m-1}\\
 &=
 mF^{m-1}\sum\phi_adq^a
     \wedge\sum\dd{p_b}{\mathcal{H}}dp_b 
     \wedge(\sum dq^c\wedge d p_c)^{m-1} \wedge(\Om^F)^k\\
 &= 
 \by{mF^{m-1}}{(m-1)!}
   \sum\phi_a\dd{p_a}{\mathcal{H}}dq^1\wedge dp_1\wedge
     \ldots\wedge dq^m\wedge dp_m \wedge(\Om^F)^k.
\end{align*}
Since $\phi_a$ depends only on $q$ and $\Ham$ is regular it
follows that $\phi_a=0$.
Because $\psi=0$ in (\ref{e:d(fsigma)}) this finishes the proof.
\end{proof}

\begin{lemma}\label{l:delta-sigma}
Under the assumptions of Theorem~\ref{thm:mult},
$\delta\sigma_A$ 
is basic with respect to the projection
$\mathcal{W}\to T^*B\to B$.
\end{lemma}

\begin{proof}
We use local Darboux coordinates $q^a,p_a$ on $T^*B$ and coordinates
$f^i$ on $F$. According to \eqref{e:sa} we may
write $\sigma_A$ terms of 
\begin{equation}\label{e:sa-loc}
 \Om^B = \sum dq^a\wedge dp_a,
 \quad
 \vv<J_F,\curv^A> = \sum\Xi_{ab}dq^a\wedge dq^b,
 \quad
 \Lam_0 = \sum\Lam_{ab}dq^a\wedge dq^b,
 \quad
 \Om^F = \sum\Om^F_{ij}df^i\wedge df^j.
\end{equation}
Let us write $\delta\sigma_A$ as 
\[
 \delta\sigma_A =
 \sum(C_a(q,p,f)dq^a + C^a(q,p,f)dp_a + D_i(q,p,f)df^i).
\]
We need to show that $C^a=0$, $D_i=0$ and $C_a=C_a(q)$.
Using the relation 
\[
 d\sigma_A
 =
 d\Lam_0
 =
 \by{1}{m-1}\delta\sigma_A\wedge\sigma_A,
\]
expanding it in terms of \eqref{e:sa-loc},
and inserting a pair $\dd{p_a}{},\dd{p_b}{}$ of vertical vectors on
both sides we see that $C^a=0$ for all $a$. Similarly one sees that
$D_i=0$. Now we insert vectors
$\dd{q^b}{},\dd{q^a}{},\dd{p_a}{}$ on both sides, and see that
$C_a(q,p)
= d_v\Lam_{ba}(\dd{p_a}{}) = C_a(q)$. (It is here that we use that $\Lam$ is linear in the
fiber.) 
\end{proof}


\section{Chaplygin systems associated to semisimple Lie
  groups}\label{sec:chap-based-on-lie} 

We associate a Chaplygin type system to a Cartan decomposition 
(and choice of a restricted root system)
of an
arbitrary (real) semisimple Lie group. In Section~\ref{sub:SO}
it is shown that this construction generalizes the
classical $n$-dimensional Chaplygin ball system.
For background on semi-simple Lie groups we refer to Knapp~\cite{K02}.

\subsection{Configuration space and constraints}

Let $G$ be a semisimple Lie group with Lie algebra $\gu$ and Killing
form $B$. Consider a Cartan decomposition $\gu = \ko\oplus\po$ 
associated to the Cartan involution $\theta$,
and let
$G\cong K\times\po$, $g=k\exp x \leftmapsto(k,x)$ be the
corresponding decomposition of the group. Thus:
\[
[\ko,\ko]\subset\ko,\qquad
[\ko,\po]\subset\po,\qquad
[\po,\po]\subset\ko.
\]
Fix a maximal abelian subspace $\ao\subset\po$, and put
$\mo=Z_{\ko}(\ao)$ and $M=Z_K(\ao)$. 
Fix also an element
$w_0\in\ao$.\footnote{This corresponds to the vertical vector orthogonal to
 the table in the case of the $n$-dimensional Chaplygin ball.} 
Define $Z_K(w_{0}) = H$ to be the stabilizer of this vector, and
note that
\begin{equation}\label{e:w}
 \ad(w_{0})|\ho^{\bot}: \ho^{\bot}:=\ho^{B\bot}\cap\ko
 \longto
 \ad(w_0)(\ho^{\bot}) =: V\subset\po
\end{equation}
is an isomorphism onto its image $V$. Of course, if $w_0$ is regular
then $H=M$ and $V=\ao^{\bot}\cap\po$.

The configuration space is now defined to be
\[
 Q := K\times V.
\]
The Lagrangian is the natural kinetic energy Lagrangian $L$ which is
associated to the positive definite inner product
$B_{\theta} = -B(.,\theta.) = -B|\ko + B|V$
taking into account the inertia tensor which is a symmetric positive
definite endomorphism $\mathbb{I}$ of $(\ko,-B|\ko)$.
Thus
\[
 L 
 = 
 \by{1}{2}\vv<\mathbb{I}u,u> + \by{1}{2}\vv<x',x'>
\]
where $\vv<.,.>=B_{\theta}$.
This Lagrangian is \emph{left}-invariant (i.e., invariant with respect
to left multiplication of $K$ on the first factor of $Q$) since we identify
$TK=K\times\ko$ via the \emph{left} multiplication, $u=s^{-1}s'$.

The distribution is 
\[
 \D = \set{(s,u,x,-\A_s(u))} \subset TK\times TV 
\]
where 
\begin{equation}\label{e:A}
 \A: (s,u)\longmapsto -[\Ad(s)u,w_{0}] = -\pr_V([\Ad(s)u,w_{0}]),\;
 TK\longto V
\end{equation}
and $w_{0}$ has been fixed to define the isomorphism \eqref{e:w}.

$(Q,\D,L)$ is a $V$-Chaplygin system with
abelian Lie group $V$. This precisely means that $(Q,\D,L)$ is a
non-holonomic system which is invariant under the free and proper  
action of the
abelian Lie group $V$ and that the distribution $\D$ determines a
principal bundle connection on $Q\toto Q/V$. The following are
essential observations.
\begin{enumerate}[\up (1)]
\item
$\A: TK\to V$ is the connection form associated to $\D$ on the principal fiber bundle
$V\hookto Q\toto K$.
\item
$\A$ is right invariant.
\end{enumerate}
The group $H=\set{h\in K: \Ad(h)w_0=w_0}$ acts through \emph{two}
different actions on $Q$:
\begin{enumerate}[\up (1)]
\item[\up (3)]
The $l$-action: 
$l_h(s,x) = (hs,x)$. This action generates \emph{internal} symmetries:
$\A\zeta^l_Y = 0$ for all $Y\in\ho$. 
($\zeta_Y^l(s)=\Ad(s^{-1}).Y$) 
\item[\up (4)]
The $d$-action: $d_h(s,x) = (hs,hx)$.
This action generates \emph{external} symmetries.
$\A(hs,u) = h.\A(s,u)$ for all $h\in H$. Thus $\D$ is invariant under
the $d$-action.
\end{enumerate}
This should be compared to the set-up in \cite{HG09}.

\subsection{Non-holonomic reduction: The compressed system}

Compression refers to the passage from the non-holonomic system $(Q,\D,L)$
with (external) symmetry group $V$
to an almost Hamiltonian system $(T^*(Q/V),\Omnh,\Hamc)$. 
Identify $T^*K=TK$ via the induced metric $\mu_0$.
According to general results on compression in the presence of internal
symmetries (e.g., \cite{EKMR04,HG09,BS93,K92}):

The compressed Hamiltonian is 
\[
 \Hamc(s,u)
 =
 \by{1}{2}\vv<\mathbb{I}u,u>+\by{1}{2}\vv<\A_s(u),\A_s(u)>
\]
which is $H$-invariant. 
The compressed almost symplectic form is
\[
 \Omnh = \Om^K-\vv<J_V\circ\textup{hl}^{\mathcal{A}},\curv_0^{\mathcal{A}}>_V
  = \Om^K + \vv<\A,d\A>_V
\]
which is also $H$-invariant.
The dynamics are given by $\Xnh$:
\[
 i(\Xnh)\Omnh = d\Hamc.
\]
Finally, according to the non-holonomic Noether Theorem 
there is a conserved quantity: 
\[
 J_H: TK\to\ho^*
\]
which is the standard \momap. 

What about reduction?
Can this data be reproduced on a quotient of the form
$J_H^{-1}(\lam)/H_{\lam}$ for some value $\lam\in\ho^*$. 
Just like in, e.g., \cite{HG09} the problem that arises is that $J_H$
is (for $w_0\neq0$) not a \momap with respect to $\Omnh$. Thus the restriction
of $\Omnh$ to a level set $J_H^{-1}(\lam)$ is not horizontal with
respect to the induced action of the stabilizer subgroup $H_{\lam}$. We
will return to this problem in Section~\ref{sec:trunc}.

\subsection{Example: $\SO(p,q)$ and Chaplygin's ball}\label{sub:SO}

Let $G=\SO(p,q)_0$ with $p\ge q$. Then the spaces under consideration
are the following.
\begin{align*}
 K 
 &=
 \set{\diag{A,D}: A\in\SO(p), D\in\SO(q)}\\
 \po
 &=
 \set{
  \left(
    \begin{matrix}
      0_{p\times p} & b\\
      b^t & 0_{q\times q}
    \end{matrix}
  \right):
  b\in\gl(p\times q,\R)}
\end{align*}
and
\begin{align*}
 \ao
 &=
 \set{ 
    \left(
      \begin{matrix}
        0_{p\times p} & b\\
        b^t & 0_{q\times q}
      \end{matrix}
    \right):
    b \textup{ has only lower antidiagonal non-zero}
  }
 = \R^q\\
 M
 &=
 \set{\diag{\SO(p-q),\theta_q,\dots,\theta_1,\theta_1,\dots,\theta_q}:
   \theta_i = \pm 1, \Pi\theta_i=1}
 = 
 \SO(p-q)\times\set{\pm1}^{q-1}.
\end{align*}
Therefore,
\begin{align*}
 K/M
 &= 
 (\SO(p)/\SO(p-q)\times\SO(q))/\set{\pm1}^{q-1}
 \cong
 V(q,p)\times\SO(q)/\set{\pm1}^{q-1}
\end{align*}
which is the ultimate reduced configuration space.

\subsubsection{Special case $q=1$, $p\ge3$}
In this case there is only one positive root and assuming that
$w_0\neq0$ yields the following. 
\begin{align*}
 K &= \SO(p)\times\set{1}\\
 \po
 &=
 \set{
  \left(
    \begin{matrix}
      0_{p\times p} & b\\
      b^t          & 0
    \end{matrix}
  \right):
  b\in\gl(p\times 1,\R) = \R^p}\\
 \ao &\cong \R^1 \textup{ and } V=\ao^{\bot}\cong\R^{p-1}\\
 H &= M \cong \SO(p-1)
\end{align*}
Thus,
\[
 \gu
 =
 \left(
 \begin{matrix}
   \mathfrak{so}(p) & \mathbb{R}^p\\
   (\mathbb{R}^p)^* & 0
 \end{matrix}
 \right)
\textup{
and
}
 w_0 
 :=
 \left(
 \begin{matrix}
  0     &  e_p\\
  e_p^t &  0
 \end{matrix}
 \right)
 \in
 \ao\subset\gu
\]
yield
\[
 \A_s(u)
 = -\textup{pr}_V[\Ad(s)u,w_0]
 =
 \left(
 \begin{matrix}
 0               &  -(\Ad(s)u).e_p\\
 -((\Ad(s)u).e_p)^t &  0
 \end{matrix}
 \right)
 \in V
\]
which can be identified with the connection form
\[
 T\SO(p)\longto\R^{p-1},
 \text{ }
 (s,u)
 \longmapsto
 -\textup{pr}_{\mathbb{R}^{p-1}}\Big((\Ad(s)u).e_p\Big)
\]
describing the
$p$-dimensional Chaplygin system when mass and radius of the ball are
both
set to $1$. See \cite{FK95,EKMR04,HG09}.
Moreover,
\[
 K/M = V(1,p) = S^{p-1}
\]
whence we recover the $p$-dimensional Chaplygin ball. (The
Lagrangian $L$ also identifies in the expected way.)

\subsection{Describing the system}\label{sec:description}
In this section we introduce notation and formulae that will be used
very much in the subsequent.
Let $\Sigma$ be the set of restricted roots associated to the pair
$(\gu,\ao)$ and $\Sigma_+\subset \Sigma$ a choice of positive
roots. 
Then the associated root space decomposition is 
\[
 \gu
 =
 \gu_0\oplus\oplus_{\lam\in\Sigma}\gu_{\lam}
 \textup{ where }
 \gu_{0} = \mo\oplus\ao.
\]
Moreover, we choose an orthonormal system
\[
 Y_{\alpha}, \; \alpha = 1,\ldots, \dim \mo
 \textup{ and }
 Z_{(\lam,a)}, \;
 \lam\in\Sigma_+,\, a = 1,\dots,\dim\gu_{\lam} 
\]
that is adapted to the decomposition $\ko=\mo\oplus\mo^{\bot}$, and an
orthonormal basis
\[ 
 e_{(\lam,a)},\;
 \lam\in\Sigma_+,\,
 a = 1,\ldots,\dim\gu_{\lam}
\]
of $\ao^{\bot}\cap\po$. We assume further the relations
\begin{equation}\label{e:rel}
 \ad(w)Z_{(\lam,a)}
 =
 \lam(w)e_{(\lam,a)}
 \textup{ and }
 \ad(w)e_{(\lam,a)} 
 =
 \lam(w)Z_{(\lam,a)}
\end{equation}
for all $w\in\ao$.
Such a basis always exists. In the following  we will use the
convention that $\alpha,\beta,\gamma,\dots$ take values
$1,\dots,\dim\mo$, and pairs $(\lam,a),(\mu,b),(\nu,c)$ have their
first component in $\Sigma_+$ while the second component runs from $1$
to the dimension of the corresponding root space. 
The basis vectors $Y_{\alpha}$, $Z_{(\lam,a)}$ as well as their dual
basis are right extended to give a right invariant frame and coframe
\[
 \xi_{\alpha}, 
 \zeta_{(\lam,a)}
 \textup{ and }
 \rho^{\alpha},\eta^{(\lam,a)}
\]
of $K$. With respect to the left trivialization this frame and coframe
becomes
\[
 \xi_{\alpha}(s) 
 = \Ad(s^{-1})Y_{\alpha} 
 = s^{-1}Y_{\alpha} 
 \textup{ and }
 \rho^{\alpha}(s)(u) 
 = \vv<\Ad(s^{-1})Y_{\alpha},u> 
 = \vv<s^{-1}Y_{\alpha},u>, 
\]
etc.  
(We will often suppress the $\Ad$-notation and simply write $s^{-1}Y$
for $\Ad(s^{-1})Y$.) 
It will be convenient to use the
notation 
\[
 l_{\alpha} = \rho^{\alpha}: TK\to\R
 \textup{ and }
 g_{(\lam,a)} = \eta^{(\lam,a)}: TK\to\R
\]
when we view the $1$-forms as functions on the tangent bundle.
These functions are the components of the angular velocity of the
ball with respect to the space frame. 
Thus the component of $\Xnh$ which is tangent to the group can
be written as
\begin{equation}\label{e:Xnh}
 T\tau.\Xnh = \sum l_{\alpha}\xi_{\alpha} + \sum g_{(\lam,a)}\zeta_{(\lam,a)}
\end{equation}
where $\tau: TK = K\times\ko\to K$. Moreover, it will be convenient to
have the notation 
\[ 
 G_{(\lam,a)} := g_{(\lam,a)}\circ\mu_0: TK\longto\R
\]
where we view $\mu_0$ as a bundle endomorphism $TK = K\times\ko\to
K\times\ko^* =_{\langle.,.\rangle} K\times\ko$.
The Liouville one-form can now be written as
\[
 \theta^K
 =
 \sum l_{\alpha}\rho^{\alpha} + \sum G_{(\lam,a)}\eta^{(\lam,a)}.
\]

With this notation we derive the following simple formula for the connection
form $\A$ which will be central to the subsequent. Namely,
\begin{equation}\label{e:A}
 \A
 =
 \sum_{\lam\in\Phi}\lam(w_0)\eta^{(\lam,a)}e_{(\lam,a)}
\end{equation}
where 
\begin{equation}\label{e:Phi}
 \Phi
 :=
 \set{\lam\in\Sigma_+: 
      \lam(w_0)\neq0}
\end{equation}
is the set of relevant roots.
For reference we also note that 
\[
 \ho=\mo\oplus\oplus_{\lam(w_0)=0}\,\textup{span}\set{Z_{(\lam,a)}}.
\]
This subalgebra is reminiscent of the $\ko$-part of the Langlands
decomposition of a parabolic subalgebra of $\gu$.
Indeed the possible
choices of $\Phi$ correspond in a one-to-one fashion 
to the possible parabolics in $\gu$.
In fact, according to Knapp~\cite[Section~VII.7]{K02} every parabolic
is specified by a set $\Gamma\subset\Sigma$ which
contains $\Sigma_+$. The correspondence is now given by setting
$\Gamma = \Sigma\setminus(-\Phi)$.
Equivalently $\Gamma$ can be defined by requiring
the identity
$-(\Gamma\cap\Sigma_-) = \Sigma_+\setminus\Phi$.
We will make use of this observation in Section~\ref{sec:rubber}.

The induced metric becomes in this notation
\[
 \mu_0
 = 
 \vv<\mathbb{I}u_1,u_2>
 + \sum_{\lam\in\Phi}\lam(w_0)^2\eta^{(\lam,a)}\otimes\eta^{(\lam,a)},
\]
which may be alternatively considered as an endomorphism 
\[
 \mu_0 
 = \mathbb{I} + \A^*\A
 = \mathbb{I} + \sum\lam(w_0)^2g_{(\lam,a)}\zeta_{(\lam,a)}
\]
of $TK=K\times\ko$.
The compressed Hamiltonian is 
\[
 \Hamc(s,u)
 =
 \by{1}{2}\vv<\mathbb{I}u,u> 
 + \by{1}{2}\sum_{\lam\in\Phi}\lam(w_0)^2g_{(\lam,a)}(s,u)^2.
\]
Furthermore,
\[
 \Omnh
 =
 \Om^K + \vv<\A,d\A>
 =
 \Om^K + \sum_{\lam\in\Phi}\lam(w_0)^2g_{(\lam,a)}d\eta^{(\lam,a)}
\]
and a formula for $d\eta^{(\lam,a)}$ is given in \eqref{e:d-eta}.

\begin{lemma}\label{lem:hor}
$\vv<\A,d\A>(\Xnh,\zeta_Y) = 0$ for all $Y\in\ho$.
\end{lemma}

\begin{proof}
This follows either form direct calculation using the above
formula. Alternatively one can use that $\Hamc$ is $H$-invariant and
that $J_H$ is a preserved quantity. Thus $\Omnh(\Xnh,\zeta_Y) = 0 =
-\Om^K(\Xnh,\zeta_Y)$. 
\end{proof}

The structure constants are of course defined by
$c^{\alpha}_{(\lam,a)(\mu,b)} =
\vv<Y_{\alpha},[Z_{(\lam,a)},Z_{(\mu,b)}]>$ etc.

\begin{lemma}\label{lem:lam-mu}
Let $\lam,\mu,\nu\in\Sigma_+$ and $1\le\alpha\le\dim\mo$.
\begin{enumerate}[\up (1)]
\item
If
$c^{\alpha}_{(\lam,a)(\mu,b)}\neq 0$ then
$\lam=\mu$. 
\item
If $c_{(\mu,b)(\nu,c)}^{(\lam,a)}\neq0$ then $\lam=\pm\mu\pm\nu$.
\end{enumerate}
\end{lemma}

\begin{proof}
To see this one notices that the $Z_{(\lam,a)}$ can be written as 
$Z_{(\lam,a)} = -X_{-\lam}^a - \theta X_{-\lam}^a\in\ko$ for a suitably
normalized orthogonal basis $X_{\lam}^a$ of $\gu$ consisting of root
vectors. (Recall that $\theta$ denotes the Cartan involution.)
The assertions now follow directly from the properties of the the root system
with respect to the action of the Lie bracket together with the fact
that $Y_{\alpha}\in\mo=\gu_0\cap\ko$. 
\end{proof}

Taking into account the change of sign in the map $\zeta:
\ko\to\X(K)$, $[X,Y]\mapsto\zeta_{[X,Y]}=-[\zeta_X,\zeta_Y]$ we obtain
the formulas
\begin{align}
 d\rho^{\alpha}
 &= 
 \by{1}{2}\sum c_{\beta\gamma}^{\alpha}\rho^{\beta}\wedge\rho^{\gamma}
 + \by{1}{2} \sum c_{(\lam,a)(\mu,b)}^{\alpha}\eta^{(\lam,a)}\wedge\eta^{(\mu,b)},\notag\\
 d\eta^{(\lam,a)}
 &=
 \sum
  c^{(\lam,a)}_{\beta(\lam,b)}\rho^{\beta}\wedge\eta^{(\lam,b)}
 + \by{1}{2}
  \sum c_{(\mu,b)(\nu,c)}^{(\lam,a)}\eta^{(\mu,b)}\wedge\eta^{(\nu,c)}.\label{e:d-eta}
\end{align}


\subsection{The preserved measure}

The $n$-dimensional Chaplygin ball problem has a preserved measure which was 
found by Fedorov and Kozlov~\cite{FK95}.
We consider the Chaplygin system $(TK,\Omnh,\Hamc)$ introduced above 
and show that the existence of a preserved
measure continues to hold.

Let $d=\dim K$ and
$g := \det\mu_0$ where we view $\mu_0$ as a function
$K\to\textup{End}(\ko)$. 
Consider the volume form 
\[
 \vol
 =
 \vol(\mu_0\times\vv<.,.>)
 =
 \by{1}{d!}g\Om^d
\]
on $TK=K\times\ko$.

\begin{lemma}\label{lem:div}
Let $f: K\to\R_{>0}$.
Then 
\[
 L_{X_{\textup{nh}}}(f(\Om^K)^d) 
 =
 d!L_{X_{\textup{nh}}}(fg^{-\frac{1}{2}}\vol)
 =
 0
 \iff
 d(\log f)\Xnh  
 =
 -\sum\dd{p_i}{}\vv<J,K>(\Xnh,\dd{q^i}{})
\]
where $(q^i,p_i)$ are canonical coordinates on $TK$.
\end{lemma}

\begin{proof}
$L_{X_{\textup{nh}}}(fg^{-\frac{1}{2}}\vol)
 = d(fg^{-\frac{1}{2}}).\Xnh\vol +
 fg^{-\frac{1}{2}}\d_{X_{\textup{nh}}}\vol$.
Thus $f$ is a preserved density corresponding to the volume $(\Om^K)^d=\Omnh^d$ iff 
\[
 d(\log f).\Xnh =
 -\d_{X_{\textup{nh}}} + \by{1}{2}d(\log g).\Xnh. 
\]
Now,
\[
 \d_{X_{\textup{nh}}}
 =
 \sum\big(
  \dd{q^i}{}(\dd{p^i}{\mathcal{H}_{\textup{c}}}
   +
   \dd{p_i}{}(-\dd{q^i}{\mathcal{H}_{\textup{c}}}+\vv<J,K>(\Xnh,\dd{q^i}{}))
  \big)
  + \by{1}{2}d(\log g).\Xnh
\]
where we use the general formula for the divergence and, of course,
the equations of motion of the almost Hamiltonian system. 
\end{proof}

By \eqref{e:Xnh} we can identify 
$d(\log f)\Xnh = \tau^*d(\log f)(\Xnh)$ with the function
$TK\to\R$ that corresponds to the one-form $d(\log f)$ on $K$.
In particular, $f$ is unique up to multiplication by positive
constants.
We will use the notation 
\[ 
 f := \by{1}{\sqrt{g}}
\]
and refer to this (after Proposition~\ref{thm:p-meas}) 
as the preserved density of the system. 
When $G=\SO(n,1)$ and we are dealing with the $n$-dimensional
Chaplygin ball then $f$ coincides with the density found by
\cite{FK95}. 
Using the rule for the differential of the determinant,
$\zeta_{(\lam,a)}\det\mu_0 = \det(\mu_0)\textup{Tr}(\mu_0^{-1}\zeta_{(\lam,a)}\mu_0)$,
one obtains
\begin{equation}\label{e:df}
 d(\log f).\zeta_{(\lam,a)}
 =
 -\sum_{(\mu,b)}\mu(w_0)^2\vv<\mu_0^{-1}[\zeta_{(\lam,a)},\zeta_{(\mu,b)}],\zeta_{(\mu,b)}>
\end{equation}
where the notation is as in Section~\ref{sec:description}.

\begin{proposition}[The preserved measure]\label{thm:p-meas}
$L_{X_{\textup{nh}}}(f(\Om^K)^d)=0$.
\end{proposition}

\begin{proof}
Of course, we will use Lemma~\ref{lem:div}.
Choose coordinates $q^i$ with $i\in J\cup I$ around a point in $K$
such that $\dd{q^i}{}(s)=\xi_{\alpha}$ for all $i\in J$ where $i$
corresponds to $\alpha$,
and
$\dd{q^i}{}(s) = \zeta_{(\lam,a)}(s)$ for all $i\in I$ where $i$
corresponds to $(\lam,a)$. 
The conjugate momenta corresponding to $i = (\lam,a)$ are then given
by $\dd{p_i}{} = (0,\mu_0^{-1}\zeta_{(\lam,a)})$. 
The first equality in the following calculation uses
Lemma~\ref{lem:hor}.
\begin{align*}
 \sum_{i\in I\cup J}\dd{p_i}{}&\vv<J,K>(\Xnh,\dd{q^i}{})
 = 
  \sum\dd{p_{(\lam,a)}}{}\vv<J,K>(\Xnh,\dd{q^{(\lam,a)}}{})\\  
 &=
 \sum\dd{p_{(\lam,a)}}{}\mu(w_0)^2g_{(\mu,b)}d\eta^{(\mu,b)}
   (\sum(l_{\alpha}\xi_{\alpha}+g_{(\nu,c)}\zeta_{(\nu,c)}),\dd{q^{(\lam,a)}}{})\\
 &=
 -\sum\dd{p_{(\lam,a)}}{}\mu(w_0)^2g_{(\mu,b)}c_{\alpha(\lam,a)}^{(\mu,b)}l_{\alpha}
 -\sum\dd{p_{(\lam,a)}}{}\mu(w_0)^2g_{(\mu,b)}c_{(\nu,c)(\lam,a)}^{(\mu,b)}g_{(\nu,c)}\\
 &=
 -
 \sum\mu(w_0)^2 \vv<\zeta_{(\mu,b)},\mu_0^{-1}\zeta_{(\lam,a)}> l_{\alpha}c_{\alpha(\lam,a)}^{(\mu,b)}
 - 
 \sum\mu(w_0)^2
   g_{(\mu,b)}\vv<\xi_{\alpha},\mu_0^{-1}\zeta_{(\lam,a)}>c_{\alpha(\lam,a)}^{(\mu,b)}\\
 &\phantom{= .}
 -
 \sum\mu(w_0)^2\vv<\zeta_{(\mu,b)},\mu_0^{-1}\zeta_{(\lam,a)}>g_{(\nu,c)}c_{(\nu,c)(\lam,a)}^{(\mu,b)}
 -
 \sum\mu(w_0)^2g_{(\mu,b)}\vv<\zeta_{(\nu,c)},\mu_0^{-1}\zeta_{(\lam,a)}>c_{(\nu,c)(\lam,a)}^{(\mu,b)}\\
 &= 
 \sum\mu(w_0)^2 g_{(\mu,b)}\vv<[\zeta_{(\mu,b)},\zeta_{(\lam,a)}]^{\xi},\mu_0^{-1}\zeta_{(\lam,a)}>
 +
 \sum\mu(w_0)^2
   g_{(\nu,c)}\vv<\zeta_{(\mu,b)},\mu_0^{-1}[\zeta_{(\nu,c)},\zeta_{(\mu,b)}]^{\zeta}>\\
 &=
 \sum\mu(w_0)^2g_{(\lam,a)}\vv<[\zeta_{(\lam,a)},\zeta_{(\mu,b)}],\mu_0^{-1}\zeta_{(\mu,b)}>
 =
 -d(\log f)\Xnh.
\end{align*}
where we have used that $c_{\alpha(\lam,a)}^{(\mu,b)} =
c_{\alpha(\lam,a)}^{(\mu,b)}\delta_{\lam\mu}$. Further, $(\_)^{\xi}$,
$(\_)^{\zeta}$ denote the projections onto the subspaces spanned by
$\xi_{\alpha}$, $\zeta_{(\lam,a)}$ respectively. Finally note that
$f$ is a pull-back of a function on the base $K$ and we have made use
of some formulas from Section~\ref{sec:description}.
\end{proof}

\textbf{Remark.}
When $\D$ is mechanical, that is orthogonal to the vertical bundle via
$\mu$, then we know that compression equals symplectic reduction at
$0$. 
(This case can be realized by setting $w_0=0$.)
Thus $\Xnh$ is the reduced Hamiltonian vector field and as such
it preserves $(\Om^K)^d$. This is consistent with the above since, now,
$J=0$ whence
$\dd{p_i}{}\vv<J,K>(\Xnh,\dd{q^i}{}) = 0$
and thus $\d_{\mu_0}\Xnh = \by{1}{2}d(\log g)\Xnh$. This can be used
as a roundabout way to reach the obvious conclusion $f=1$.


\subsection{Truncation}\label{sec:trunc}

The system $(TK,\Omnh,\Hamc)$ is $H$-invariant and has a preserved
quantity which is just the standard \momap $J_H: TK\to\ho^*$. Thus it
is natural to ask whether this set of data can be reduced to
$J_H^{-1}(\orb)/H\cong J_H^{-1}(\alpha)/H_{\alpha}$ where $\orb$ is an
$\Ad^*(H)$-orbit through $\alpha\in\ho^*$ and $H_{\alpha}$ is the
stabilizer of $\alpha$ in the group. The answer to this question
is negative: the \momap equation 
\[ 
 i(\zeta_Y)\Omnh = d\vv<J_H,Y>
\]
with $Y\in \ho$ is \emph{not} satisfied in general. Thus the restriction
of $\Omnh$ to $J_H^{-1}(\alpha)$ is not horizontal in general whence
it cannot induce a form on the reduced space. The situations here is
of course identical with that of \cite{HG09}. 
Thus by \cite[Theorem~3.3]{HG09} we also know that there is a
solution: the form $\vv<J,K>$ is not optimal for describing the
system; it sees vertical directions that are inessential
(Lemma~\ref{lem:hor}) whence it needs to be replaced by an entity
which is horizontal.

Let 
\begin{align}\label{e:Lam}
 \Lam\notag
 :=
 &-
 \by{1}{2}\sum\lam(w_0)^2c^{\alpha}_{(\lam,a)(\lam,b)}l_{\alpha}\eta^{(\lam,a)}\wedge\eta^{(\lam,b)}
 -
 \by{1}{2}\sum_{\lam\notin\Phi,\mu,\nu\in\Phi}
   \mu(w_0)^2c_{(\mu,b)(\nu,c)}^{(\lam,a)}g_{(\lam,a)}\eta^{(\mu,b)}\wedge\eta^{(\nu,c)}\\
 &+
 \by{1}{2}\sum_{\mu,\nu\in\Phi}\lam(w_0)^2c_{(\mu,b)(\nu,c)}^{(\lam,a)}g_{(\lam,a)}\eta^{(\mu,b)}\wedge\eta^{(\nu,c)}.
\end{align}
Notice that the coefficients of the second summand of $\Lam$ are
skew-symmetric: when $c^{(\lam,a)}_{(\mu,b)(\nu,c)}\neq0$ with
$\lam\notin\Phi$ and $\mu,\nu\in\Phi$ then $\mu(w_0)^2=\nu(w_0)^2$ by
Lemma~\ref{lem:lam-mu}. 
Of course, one makes a choice here: in principle 
one could add to $\Lam$ any $\tau$-semi-basic
$H$-basic two-from which vanishes upon contraction with $\Xnh$.
However, in the proof of Theorem~\ref{thm:ham-at-0} we will see that
this choice for $\Lam$ seems to be preferred by the problem at hand.

The following theorem generalizes \cite[Theorem~4.1]{HG09}.

\begin{theorem}[Truncation]\label{thm:trunc}
The system $(TK,\wt{\Om},\Hamc)$ where 
\[
 \wt{\Om} := \Om^K+\Lam
\]
has the following properties.
\begin{enumerate}[\up (1)]
\item
$\wt{\Om}$ is non-degenerate and $H$-basic.
\item
$i(\Xnh)\wt{\Om} = d\Hamc$.
\item
$i(\zeta_Y)\wt{\Om} = d\vv<J_H,Y>$ for all $Y\in\ho$. 
\end{enumerate}
\end{theorem}

\begin{proof}
Non-degeneracy is clear.
Observe that
\begin{align*}
 \Big(
 &\by{1}{2}\sum\lam(w_0)^2c^{\alpha}_{(\lam,a)(\lam,b)}l_{\alpha}\eta^{(\lam,a)}\wedge\eta^{(\lam,b)}
 +
 \by{1}{2}\sum_{\lam\notin\Phi,\mu,\nu\in\Phi}
    \mu(w_0)^2c_{(\mu,b)(\nu,c)}^{(\lam,a)}g_{(\lam,a)}\eta^{(\mu,b)}\wedge\eta^{(\nu,c)}
 \Big)_{(s,u)}(u_1',u_2')\\
 &=
 \vv<[\ad(w_0)^2s.u_1',s.u_2']^{\mathfrak{h}},s.u>
\end{align*}
where $(\_)^{\mathfrak{h}}$ denotes projection onto $\ho$. Clearly
this is $H$-invariant since, by definition, $H$ commutes with
$\ad(w_0)$.  
On the other hand,
\[
 \Big(
 \by{1}{2}\sum_{\mu,\nu\in\Phi}\lam(w_0)^2c_{(\mu,b)(\nu,c)}^{(\lam,a)}g_{(\lam,a)}\eta^{(\mu,b)}\wedge\eta^{(\nu,c)}
 \Big)_{(s,u)}(u_1',u_2')
 =
 \vv<[su_1',su_2']^{\mathfrak{h}^{\bot}},\ad(w_0)^2su>
\]
which is also $H$-independent. Thus $\Lam$ is $H$-invariant. Obviously
$\Lam$ is also $H$-horizontal since the $\eta^{(\mu,b)}$ for
$\mu\in\Phi$ are horizontal by construction.
To see that $\wt{\Om}$ produces the right dynamics note simply that 
\begin{align*}
 \vv<\A,d\A>(\Xnh,\zeta_{(\nu,c)})
 =
 {}&\sum \mu(w_0)^2g_{(\mu,b)}l_{\alpha}c^{(\mu,b)}_{\alpha(\nu,c)}\delta_{\mu,\nu}
 + 
 \sum
 \lam(w_0)^2g_{(\lam,a)}g_{(\mu,b)}c^{(\lam,a)}_{(\mu,b)(\nu,c)}\\
 =
 {}&-\sum
   \mu(w_0)^2g_{(\mu,b)}l_{\alpha}c^{\alpha}_{(\mu,b)(\nu,c)}\delta_{\mu,\nu}
 -\sum_{\mu\notin\Phi}
   \lam(w_0)^2g_{(\lam,a)}g_{(\mu,b)}c^{(\mu,b)}_{(\lam,a)(\nu,c)}\\
 &+ 
 \sum_{\mu\in\Phi}
 \lam(w_0)^2g_{(\lam,a)}g_{(\mu,b)}c^{(\lam,a)}_{(\mu,b)(\nu,c)}\\
 =
 {}&\Lam(\Xnh,\zeta_{(\nu,c)})
\end{align*}
for all $\nu\in\Phi$.
Finally, we can use the \momap equation with respect to $\Om^K$ and
horizontality of $\Lam$ to obtain the \momap equation for $\wt{\Om}$. 
\end{proof}

Thus one can pass to the description $(TK,\wt{\Om},\Hamc)$ of the system
and do (almost) Hamiltonian reduction with respect to the symmetry
group $H$ and the \momap $J_H$. Using the mechanical connection
associated to $\mu_0$ the reduced space can be realized 
as a symplectic fiber bundle
over $T^*(K/H)$ with fiber a coadjoint orbit $\mathcal{O}\subset\ho^*$
whence Theorem~\ref{thm:mult} is applicable.


\subsection{Cases of Hamiltonization for multidimensional systems}\label{sec:ham}

In this setting multidimensional means that the dimension of the
ultimate reduced configuration space $K/H$ is greater than $2$. 

By
Theorem~\ref{thm:trunc} we regard the compressed system as being
described by the almost Hamiltonian system $(TK,\wt{\Om},\Hamc)$ and
we recall that we identify $TK=T^*K$ via the induced metric
$\mu_0$. According to Theorem~\ref{thm:p-meas} this system admits a preserved
measure: 
$L_{X_{\textup{nh}}}(f\Om_K^d)=0$ where $d=\dim K$ and
\[
 f = (\det\mu_0)^{-\frac{1}{2}}.
\]
(From Lemma~\ref{lem:div} it is not hard to see that $f$ factors also
to a density on $T^*(K/H) = J_H^{-1}(0)/H$.)
Let
$\iota:
J_H^{-1}(\alpha)\hookto TK$,
$\alpha\in\ho^*$, 
$\pi: J_H^{-1}(\alpha)\toto H_H^{-1}(\alpha)/H_{\alpha}$
where $H_{\alpha}$ is the isotropy subgroup of $\alpha$ in $H$,
and 
\[
 F := f^{\frac{1}{m-1}}
\]
with 
$
 m = \by{1}{2}\dim J_H^{-1}(\alpha)/H_{\alpha}
$. 
Then
the reduced almost symplectic form $\sigma$ is characterized by the
equation $\pi^*\sigma = \iota^*\wt{\Om}$. Note that
we may use the metric $\mu_0$ to identify
\begin{equation}\label{e:red-space}
 J_H^{-1}(\alpha)/H_{\alpha}
 \cong
 J_H^{-1}(\orb)/H
 \cong
 T^*(K/H)\times_{K/H}(K\times_H\orb)
\end{equation}
where $\orb$ is the $\Ad^*(H)$-orbit through $\alpha$ and
$\sigma$ is of the
form `canonical plus magnetic plus semi-basic' with the semi-basic
part linear in the fibers whence Theorem~\ref{thm:mult} is
applicable. 
Thus, up to multiplication by positive constants, the only possible
candidate 
for a conformal factor of $\sigma$ will be $F$ which we can view as a
function $K/H_{\alpha}\to\R_{>0}$. 
(Because $\delta\sigma = -(m-1)d\log F$ in this case.)
Now it is a trivial observation to note that $F$ indeed is a conformal
factor if and only if
\begin{equation}\label{e:iff-cond}
 \iota^*d\Lam = -\iota^*(d(\log F)\wedge\wt{\Om}).
\end{equation}
Analyzing this equation for $\alpha=0$ leads to the following result.

\begin{theorem}[Hamiltonization at $0$ momentum]\label{thm:ham-at-0} 
Let $m=\dim{K/H}$. 
The induced almost symplectic structure $\sigma$ on
$J_H^{-1}(0)/H\cong T^*(K/H)$ is Hamiltonizable if and only if
the metric tensor 
$\mu_0 = \mathbb{I}+\sum\lam(w_0)^2g_{(\lam,a)}\zeta_{(\lam,a)}:
\ko\to\ko$ 
satisfies
\begin{multline}\label{e:ham-at-0}
 \langle s\mu_0(s)^{-1}s^{-1}Z_{(\kappa,d)},
 [\ad(w_0)^2Z_{(\mu,b)},Z_{(\nu,c)}]^{\mathfrak{h}} - \ad(w_0)^2[Z_{(\mu,b)},Z_{(\nu,c)}]\rangle\\
 =
 \by{1}{m-1}
  \sum\vv<s\mu_0(s)^{-1}s^{-1}Z_{(\lam,a)},
     [Z_{(\mu,b)},\ad(w_0)^2Z_{(\lam,a)}]\delta_{(\nu,c),(\kappa,d)} 
     -
     [Z_{(\nu,c)},\ad(w_0)^2Z_{(\lam,a)}]\delta_{(\mu,b),(\kappa,d)}>
\end{multline}
for all $\kappa,\mu,\nu\in\Phi$. 
Here $(\_)^{\mathfrak{h}}$ denotes the projection onto $\ho$ with
respect to the $\Ad$-invariant inner product.
As usual
$\delta_{(\nu,c),(\kappa,d)}$ is $1$
if $(\nu,c)=(\kappa,d)$ and $0$ else. Moreover, if this condition is
satisfied then 
\begin{equation}\label{e:exact}
 \pi^*(F\sigma)
 =
 \iota^*(F\wt{\Om})
 =
 -\iota^*d(F\sum G_{(\lam,a)}\eta^{(\lam,a)})
 =
 -\pi^*d(F\theta^{K/H})
\end{equation}
where $\theta^{K/H}$ is the Liouville one-form on $T^*(K/H)$.
That is, $F\sigma$ is even exact. 
\end{theorem}

We remark that $\mathbb{I}=1$ implies that
$s\mu_0(s)s^{-1}=\mu_0(e)$. 
Notice that the condition simplifies when $|\Phi|=1$ as is the case
for the $n$-dimensional Chaplygin ball. When $\dim K/H = 2$ then the
condition is empty in agreement with the Chaplygin multiplier theorem.

\begin{proof}
Let us first prove that \eqref{e:ham-at-0} implies
\eqref{e:exact}. Since $\iota^*(F\wt{\Om}) =
\iota^*(-Fd(\sum G_{(\lam,a)}\eta^{(\lam,a)}) + F\Lam)$ it suffices to show
that $\Lam = -d(\log F)\wedge\sum G_{(\lam,a)}\eta^{(\lam,a)}$ along
$J_H^{-1}(0)$.\footnote{We view this as `compelling evidence' that the
  choice for $\Lam$ in \eqref{e:Lam} is in a sense optimal.} 
Consider an element $(s,u)\in J_H^{-1}(0)$ with $u =
\mu_0^{-1}\zeta_{(\kappa,d)}$ where $\kappa\in\Phi$. (Notice that we
sometimes drop the base point $s$ in order not to make the notation
too cumbersome.) Then with
$\mu,\nu\in\Phi$ we have
\begin{align*}
 \Lam_{(s,u)}(\zeta_{(\mu,b)},\zeta_{(\nu,c)}) 
 &= 
  -\sum_{\alpha}
    \mu(w_0)^2\delta_{\mu,\nu}c^{\alpha}_{(\mu,b)(\nu,c)}\vv<Y_{\alpha}s\mu_0^{-1}s^{-1}Z_{(\kappa,d)}>\\
  &\phantom{==}
  -\sum_{\lam\notin\Phi}
    \mu(w_0)^2c_{(\mu,b)(\nu,c)}^{(\lam,a)}\vv<Z_{(\lam,a)},s\mu_0^{-1}s^{-1}Z_{(\kappa,d)}>\\
  &\phantom{==}
  +
   \sum_{\lam\in\Phi}
    \lam(w_0)^2c_{(\mu,b)(\nu,c)}^{(\lam,a)}\vv<Z_{(\lam,a)},s\mu_0^{-1}s^{-1}Z_{(\kappa,d)}>\\
 &=
  -\vv<[\ad(w_0)^2Z_{(\mu,b)},Z_{(\nu,c)}]^{\mathfrak{m}},s\mu_0^{-1}s^{-1}Z_{(\kappa,d)}>\\
  &\phantom{==}
  -\vv<[\ad(w_0)^2Z_{(\mu,b)},Z_{(\nu,c)}]^{\mathfrak{h}\cap\mathfrak{m}^{\bot}},s\mu_0^{-1}s^{-1}Z_{(\kappa,d)}>\\
  &\phantom{==}
  +
  \vv<\ad(w_0)^2[Z_{(\mu,b)},Z_{(\nu,c)}],s\mu_0^{-1}s^{-1}Z_{(\kappa,d)}>\\
 &=
  -\vv<s\mu_0^{-1}s^{-1}Z_{(\kappa,d)},
     [\ad(w_0)^2Z_{(\mu,b)},Z_{(\nu,c)}]^{\mathfrak{h}}
     - \ad(w_0)^2[Z_{(\mu,b)},Z_{(\nu,c)}]>.
\end{align*}
As before, the superscript $(\_)^{\mathfrak{m}}$ denotes projection
onto $\mo$ with respect to the $\Ad$-invariant inner product $\vv<.,.>$.
On the other hand,
\begin{align*}
 -(d(\log F)\wedge\sum G_{(\lam,c)}\eta^{(\lam,a)})_{(s,u)}(\zeta_{(\mu,b)},\zeta_{(\nu,c)})
 &=
  \by{1}{m-1}
   \sum\lam(w_0)^2
   \vv<\mu_0^{-1}[\zeta_{(\mu,b)},\zeta_{(\lam,a)}],\zeta_{(\lam,a)}>
   \delta_{(\kappa,d),(\nu,c)}\\
  &\phantom{==}
   -
   \by{1}{m-1}
    \sum\lam(w_0)^2
    \vv<\mu_0^{-1}[\zeta_{(\nu,c)},\zeta_{(\lam,a)}],\zeta_{(\lam,a)}>
    \delta_{(\kappa,d),(\mu,b)}\\
 &=
  -\by{1}{m-1}
    \sum\vv<s\mu_0^{-1}s^{-1}Z_{(\lam,a)},
        [Z_{(\mu,b)},\ad(w_0)^2Z_{(\lam,a)}]>\delta_{(\nu,c),(\kappa,d)}\\
  &\phantom{==}
  +
   \by{1}{m-1}
    \sum\vv<s\mu_0^{-1}s^{-1}Z_{(\lam,a)},
         [Z_{(\nu,c)},\ad(w_0)^2Z_{(\lam,a)}]>\delta_{(\mu,b),(\kappa,d)}.
\end{align*}
Since the two-forms in question are semi-basic and linear in the
fibers this proves that they are equal along the $0$ level set of
$J_H$. Note also that the pull-back of the Liouville one-form on $T^*(K/H)$
equals $\iota^*\sum G_{(\lam,a)}\eta^{(\lam,a)} =
\iota^*\sum_{\lam\in\Phi}G_{(\lam,a)}\eta^{(\lam,a)}$. 
To see that the condition is also necessary one evaluates 
Equation~\eqref{e:iff-cond} on a triple
of the form
$(\zeta_{(\mu,b)},\zeta_{(\nu,c)},\dd{G_{(\kappa,d)}}{} =
(0,\mu_0^{-1}\zeta_{(\kappa,d)}))$. 
The resulting calculation is very similar to the one above.
\end{proof}


\section{Examples}\label{sec:examples}

This section contains examples of the class of non-holonomic systems
introduced in the previous section. We continue all the notation from
above, most of which has been
introduced in
Section~\ref{sec:description}. 
In particular,
$\Sigma$ will be the set of restricted roots associated to a pair
$(\gu,\ao)$ and $\Sigma_+\subset \Sigma$ a choice of positive
roots. 
Then the associated root space decomposition is 
$
 \gu
 =
 \gu_0\oplus\oplus_{\lam\in\Sigma}\gu_{\lam}
$
where 
$
 \gu_{0} = \mo\oplus\ao.
$
Moreover, we choose an orthonormal system
$
 Y_{\alpha}
$
and 
$Z_{(\lam,a)}$,
that is adapted to the decomposition $\ko=\mo\oplus\mo^{\bot}$, and an
orthonormal basis
$
 e_{(\lam,a)}
$
of $\ao^{\bot}\cap\po$.
We will in each example fix an element $w_0\in\ao$ and consider the
set $\Phi := \set{\lam\in\Sigma_+: \lam(w_0)\neq0}$.


\subsection{$\SO(n,1)$, Hamiltonization of Chaplygin's ball}
According to Section~\ref{sub:SO} the above Theorem~\ref{thm:ham-at-0}
should have some bearing on the $n$-dimensional Chaplygin ball system
with angular momentum $\alpha=0$. 
Moreover, for this system there is only $1$ positive root
(and we assume that $\lam(w_0)=1$ for this root)  whence
Condition~\eqref{e:ham-at-0} simplifies to
\[
 \vv<s\mu_0^{-1}s^{-1}Z_d,[Z_b,Z_c]>
 =
 \by{1}{m-1}
 \sum_a\vv<s\mu_0^{-1}s^{-1}Z_a,
   [Z_b,Z_a]\delta_{cd}
   - [Z_c,Z_a]\delta_{bd}>.
\]
Writing this equation in terms of the inertia tensor $\mathbb{I}$
implies that the system is Hamiltonizable at the
$T^*(K/H)=T^*(\SO(n)/\SO(n-1))$-level if and only if $\mathbb{I}$ satisfies
\begin{equation}\label{e:ham-i}
 s^{-1}Z_d
 =
 (\mathbb{I}+1)s^{-1}Z(d) + \mathbb{I}\sum_b\mathcal{M}_{b,d}(s)s^{-1}[Z_b,Z_d]
\end{equation}
for an arbitrary $d$-dependent vector $Z(d)\in\ho^{\bot}$ and
arbitrary $s$- and $b,d$-dependent numbers $\mathcal{M}_{b,d}(s)\in\R$.
We will identify $\so(n)$ with $\R^n\wedge\R^n$ and hence $Z_d =
e_d\wedge e_n$ and $[Z_b,Z_d]=e_b\wedge e_d$ where $e_1,\dots,e_n$ is
the standard basis of $\R^n$. Simultaneously we revert to writing
$\Ad(s)$ for the adjoint action of $s$ on $\so(n)$. 

Making the simplifying assumption that $\mathbb{I}$ is diagonal with
respect to the basis $Y_{\alpha},Z_a$ of $\ko=\so(n)$ and evaluating
\eqref{e:ham-i} at $s=e$ then implies that
$Z(d)=(\mathbb{I}+1)^{-1}Z_d=\var_dZ_d$ for some $\var_d>0$. 
Therefore,
\[
 \mathbb{I}e_d\wedge e_n
 =
 \by{1-\var_d}{\var_d}e_d\wedge e_n.
\]
A choice of a number $a_n>0$ then induces a prescription
\[
 \var_d \mapsto \by{1-\var_d}{a_n} = a_d,
 \quad
 a_d \mapsto \var_d = 1-a_da_n
\]
which can be taken as a motivation to define
\begin{equation}\label{e:jov-i}
 \mathbb{I}e_i\wedge e_j 
 =
 \by{a_ia_j}{1-a_ia_j}e_i\wedge e_j
 \textup{ with }
 0<a_ia_j<1
 \textup{ for }
 1\le i,j \le n.
\end{equation}
This is the inertia tensor of Jovanovic~\cite[Section~4]{Jov09}. 
Another equivalent way to write \eqref{e:ham-i}
is
\begin{equation}\label{e:ham-i2}
 \mu_0^{-1}\Ad(s^{-1})(e_d\wedge e_n)
 =
 \Ad(s^{-1})Z(d) + \sum\mathcal{M}_{b,d}(s)\Ad(s^{-1})(e_b\wedge e_d)
\end{equation}
with the same notation as above. Going through the proof of Theorem~3
of \cite{Jov09} one sees that 
\begin{align*}
 \mu_0^{-1}\Ad(s^{-1})(e_d\wedge e_n)
 &=
 \vv<s^{-1}e_n,A^{-1}s^{-1}e_n>
 \Big(
  (-As^{-1}e_d + 
        \vv<A^{-1}s^{-1}e_n,s^{-1}e_n>s^{-1}e_d)\wedge s^{-1}e_n\\
  &\phantom{==}
  +
  \sum\vv<A^{-1}s^{-1}e_n,s^{-1}e_b>s^{-1}e_b\wedge s^{-1}e_d
 \Big)
\end{align*}
where $A:=\diag{a_1,\dots,a_n}$. With 
$Z(d) =  \vv<s^{-1}e_n,A^{-1}s^{-1}e_n>
  (-As^{-1}e_d + 
        \vv<A^{-1}s^{-1}e_n,s^{-1}e_n>s^{-1}e_d)\wedge s^{-1}e_n$
and
$\mathcal{M}_{b,d}(s)
 =
 \vv<s^{-1}e_n,A^{-1}s^{-1}e_n>\vv<A^{-1}s^{-1}e_n,s^{-1}e_b>$
this clearly satisfies \eqref{e:ham-i2}. Thus the system defined by
the inertia tensor \eqref{e:jov-i} is Hamiltonizable at the
$T^*(K/H)$-level which reproduces the result of
\cite[Theorem~5]{Jov09}. 
In fact, the rescaled form is given by \eqref{e:exact} whence it is not only
symplectic but even exact.


\subsection{$\SL(n,\R)$}
Let 
$\gu = \mathfrak{sl}(n,\R)$. 
Then 
$\ko = \so(n)$,
$\po = \set{x\in\mathfrak{sl}(n,\R):x^t=x}$,
$\ao = \set{\diag{w^1,\dots,w^n}\in\mathfrak{sl}(n,\R)}$,
and $\mo=\set{0}$. Thus there are no internal symmetries when $w_0$ is
regular. Let $f_i:\mo\to\R$, $w=\diag{w^1,\dots,w^n}\mapsto w^i$ for
$1\le i\le n$. Similarly to the Cartan case the restricted root system
$\Sigma = \set{\lam_{ij} := f_i-f_j: i\neq j}$ 
associated to $(\gu,\mo)$ is of type
$A_{n-1}$. A choice of a positive system is $\Sigma_+ =
\set{\lam_{ij}: i<j}$.

Let $n=3$. According to \eqref{e:A} the constraints are determined by
the connection form 
$\A: 
 TK\to V 
 = \set{x\in\mathfrak{sl}(3,\R):x^t=x \textup{ and } x^{ii} = 0}$,
\begin{equation}\label{e:Asl}
 \A:
 (s,u)\mapsto\Ad(s)u 
 = 
 \wt{u}
 =
 \left(
  \begin{matrix}
    \wt{u}^1\\
    \wt{u}^2\\
    \wt{u}^3
  \end{matrix}
 \right)
 \mapsto
 -\ad(w_0)\wt{u}
 =
 -\left(
   \begin{matrix}
     \lam_3(w_0)\wt{u}^1\\
     \lam_1(w_0)\wt{u}^2\\
     \lam_2(w_0)\wt{u}^3
   \end{matrix}
  \right)
\end{equation}
where $\lam_1=\lam_{13}>\lam_2=\lam_{12}>\lam_3=\lam_{23}$ are the
ordered positive roots.   
Note that $\lam_2+\lam_3=\lam_1$.
The basis vectors $Z_{(\lam,a)}$, $e_{(\lam,a)}$ introduced in
Section~\ref{sec:description} can now be identified with
$Z_{\lam_1}=(0,1,0)^t$, etc., considered as an element of $\ko\cong\R^3$ and 
$e_{\lam_1}=Z_{\lam_1}=(0,1,0)^t$, etc., considered as an element of
$V\cong\R^3$. 

For generic $w_0$, $Q\cong\SO(3)\times\R^3$, and the system \eqref{e:Asl} could be viewed as a three-axial
ellipsoid with constraints moving through space.
There are no internal symmetries, $\ho=\mo=0$, in this case. Using the relation
$[Z_{\lam_1},Z_{\lam_2}]=Z_{\lam_3}$ condition~\eqref{e:ham-at-0} with
$\kappa=\lam_3$, $\mu=\lam_1$ and $\nu=\lam_2$ thus becomes
$\lam_3(w_0)^2\vv<\mu_0^{-1}Z_{\lam_3},Z_{\lam_3}> = 0$.
Since $\mu_0$ is positive definite this implies $\lam(w_0)=0$
contradicting genericity of $w_0$.
Thus this case is \emph{never} Hamiltonizable, not even for the
homogeneous case $\mathbb{I}=1$. This is in contrast with the
$n$-D Chaplygin ball system \cite[Corollary~4.3]{HG09}.

However, when $\lam_2(w_0)=0$ and $\lam_1(w_0)=\lam_3(w_0)\neq0$ then
$H=S^1$ and we
recover the $3$-D Chaplygin ball system.


\subsection{$\textup{Sp}(n,\R)$}
Let $G=\textup{Sp}(n,\R) = \set{g\in\SL(2n,\R): g^tJg=J}$ where $J$ is
the standard complex structure on $\R^{2n}$. 
Thus $\gu=\mathfrak{sp}(n,\R)$ consists of matrices of the form
\[
\left(
\begin{matrix} 
 X_1 & X_2\\
 X_3 & -X_1^t
\end{matrix}
\right)
\]
with $X_i\in\gl(n,\R)$ such that $X_2$ and $X_3$ are symmetric. The
constituents of the Cartan decomposition are 
$\ko =
\so(2n)\cap\mathfrak{sp}(n,\R)\cong\mathfrak{u}(n)$,
$K=\textup{U}(n)$, 
and 
$\po = \set{x\in\gu: x^t=x}$, and $\ao$ is the subspace of diagonal
matrices in $\po$ and $\mo=\set{0}$.

For convenience we will
restrict now to the case $n=2$. 
For $i=1,2$ define $f_i\in\ao^*$ to be the mapping 
$f_i:
\diag{w^1,w^2,-w^1,-w^2}\mapsto w^i$. Then the positive restricted
roots associated to $(\gu,\ao)$ are 
\[
 \Sigma_+ = \set{f_1-f_2,f_1+f_2,2f_1,2f_2}.
\]
Note that $\set{f_1-f_2,2f_2}$ forms a simple system. Since we are
interested in having internal symmetries we fix an element
$w_0 =
\diag{a,a,-a,-a}\in\ao$ with $a>0$. Thus $(f_1-f_2)(w_0)=0$, 
$\Phi = \set{f_1+f_2,2f_1,2f_2}$ and $\lam(w_0) = 2a$ for all
$\lam\in\Phi$.
Therefore,
\[
 \A:
 (s,u)\mapsto\Ad(s)u 
 = 
 \wt{u}
 =
 \left(
  \begin{matrix}
    \wt{u}^1\\
    \wt{u}^2\\
    \wt{u}^3\\
    \wt{u}^4
  \end{matrix}
 \right)
 \mapsto
 -\ad(w_0)\wt{u}
 =
 -2a\left(
     \begin{matrix}
       0\\
       \wt{u}^2\\
       \wt{u}^3\\
       \wt{u}^4
     \end{matrix}
    \right)
\]
Further, the configuration space is
$
 Q = K\times V \cong U(2)\times\R^3
$ 
and
$
 \ko 
 = \ho\oplus\ho^{\bot}
 = 
 \set{ 
 yZ_{f_1-f_2}: y\in\R
 }
 \oplus
 \set{
 z^{11}Z_{2f_1} + z^{12}Z_{f_1+f_2} + z^{22}Z_{2f_2}
 : z^{ij}\in\R
 }
$
where
\[
 Z_{f_1-f_2}
 =
 \left(
 \begin{matrix}
  0 & -1 &   &   \\
  1 & 0  &   &   \\
    &    & 0 & -1\\
    &    & 1 & 0
 \end{matrix}
 \right)
 \textup{ and }
 z^{11}Z_{2f_1} + z^{12}Z_{f_1+f_2} + z^{22}Z_{2f_2}
 =
 \left(
 \begin{matrix}
   &      &        & z^{11} & z^{12}\\
   &      &        & z^{12} & z^{22}\\
  -z^{11} & -z^{12} &       &       \\
  -z^{12} & -z^{22} &       &       
 \end{matrix}
 \right)
\]
Notice also that one can read off from the properties of
the root system that
$[\ho^{\bot},\ho^{\bot}]\subset\ho$ whence
the left and right
hand side of \eqref{e:ham-at-0} are both identically $0$ for the
homogeneous case $\mathbb{I}=1$. Thus the homogeneous case is
Hamiltonian ($F$ is constant) at the ultimate reduced level $T^*(U(2)/S^1)$. 

For general $n$ one can use that the root system $\Sigma(\gu,\ao)$ is of type
$C_n$ whence the positive system will be of the form $\Sigma_+ =
\set{f_i\pm f_j: 1\le i<j\le n}\cup\set{2f_i: 1\le i\le n}$ and the simple roots
are $f_i - f_j$ with $1\le i<j\le n$ and $2f_n$. A choice of $w_0$ can
now be determined by letting appropriately many simple roots vanish on
$w_0$. E.g., one can conclude just as above that choosing a non-zero $w_0$ 
in the joint kernel of $f_i - f_j$ with $1\le i<j\le n$ yields a
system which is Hamiltonian at the ultimate reduced level $T^*(K/H) = T^*(U(n)/(U(1)^{n-1}))$.


\subsection{Split $G_2$, $2-3-5$, $1/3$ and rubber rolling}\label{sec:rubber}

Let $G$ be the split real form of the the exceptional complex
semi-simple Lie group $G_2$. This group is $14$-dimensional and can be
realized as the automorphism group of the split octonions. We refer to
\cite{SV00,S06,K02} for background. 
The Cartan decomposition data are the following,
\[
 K = \textup{SU}(2)\times_{(\pm1)}\textup{SU}(2) 
     \cong\SO(4), \text{ }
 \po\cong\R^8,\text{ }
 \ao\cong\R^2,
 \textup{ and }
 \mo=\set{0}.
\]
The restricted roots are of type $G_2$ whence a positive system can be
written as
\[
 \Sigma_+
 =
 \set{\lam_1,\lam_2,\lam_1+\lam_2,\lam_1+2\lam_2,2\lam_1+3\lam_2,\lam_1+3\lam_2}
\]
with 
$\lam_1$ and $\lam_2$ simple.
We choose $w_0\in\ao$ such that $\lam_1(w_0)=0$ and
$\lam_2(w_0)\neq0$. 
Thus the set of relevant roots is
$
 \Phi
 =
 \set{\lam_2,\lam_1+\lam_2,\lam_1+2\lam_2,2\lam_1+3\lam_2,\lam_1+3\lam_2}
$
and the infinitesimal internal symmetries are 
\[
 \ho = \textup{span}\set{Z_{\lam_1}} = \R
\]
which we view as the Lie algebra of the connected component $H$ of
$Z_K(w_0)$,
\[
 H \cong S^1.
\]
According to Section~\ref{sec:chap-based-on-lie} we have
$V = \ad(w_0)\ko = \textup{span}\set{e_{\lam}: \lam\in\Phi} \cong \R^5$ and therefore
\[
 Q\cong K\times\R^5
 \textup{ and }
 Q/(\R^5\times H) = K/H
 \cong \textup{SU}(2)\times\SO(3)/S^1
 \cong \textup{SU}(2)\times S^2.
\]
We remark that $K/H\cong G/P_{w_0}$
where $P_{w_0}$ is the
parabolic subgroup of $G$ associated to the subset of simple roots $\Pi$
consisting of $\set{\lam\in\Pi: \lam(w_0)=0} = \set{\lam_1}$. 

What about Hamiltonization? Suppose $\mathbb{I}=1$ which implies that
$s\mu_0(s)s^{-1}=\mu_0(e)$ and 
$\mu_0(e)^{-1}Z_{\kappa} = (1+\kappa(w_0)^2)^{-1}Z_{\kappa}$ for all
$\kappa\in\Sigma_+$. 
Thus the left hand side of \eqref{e:ham-at-0} is non-zero for, e.g.,
$\kappa=\lam_1+\lam_2$, $\mu=\lam_1+2\lam_2$ and
$\nu=2\lam_1+3\lam_2$. Thus the system
is not Hamiltonizable at the
$T(K/H)$-level corresponding to reduction of
$(TK,\wt{\Om},\Hamc)$ 
at $0$-level set of the $J_H$-\momap. 

On the other hand we recognize $K/H$ as the double cover 
configuration space $\SO(3)\times S^2$ 
of the sphere-on-sphere-rolling system. This
system is a natural generalization of the Chaplygin ball on a table
when one forbids slipping. One can also introduce a no-twist
constraint and the resulting non-holonomic system has been shown to be
Hamiltonizable by Koiller and Ehlers~\cite{EK07}. Moreover, it seems to be known since
Cartan that $G_2$ is related to this no-twist no-slip sphere-on-sphere
system. 
Therefore, 
one might expect some relation between this system and the one 
defined by 
$(TK,\wt{\Om},\Hamc)$
even though 
the non-Hamiltonizability of the latter is apparently an obstruction to
any such relation.

Recall from Theorem~\ref{thm:trunc} that $\wt{\Om}=\Om^K+\Lam$. In
order to stand a chance at obtaining a Hamiltonizable system we
consider the set
$
 \set{(s,u)\in TK: i(\Xnh)\Lam_{(s,u)}=0}.
$
By \eqref{e:Lam} we have 
\[
 i(\Xnh)\Lam(\zeta_{\nu}) 
 =  
  -\sum_{\mu}\mu(w_0)^2c^{\lam_1}_{\mu\nu}g_{\mu}g_{\lam_1} 
  +
  \sum_{\lam,\mu\in\Phi}\mu(w_0)^2c^{\lam}_{\mu\nu}g_{\lam}g_{\mu}.
\] 
Setting $\nu=\lam_1+2\lam_2$ the possibilities for $\set{\lam,\mu}$ 
are 
$\set{\lam_2,\lam_1+\lam_2}$ and $\set{\lam_2,\lam_1+3\lam_2}$. 
The resulting condition for
$i(\Xnh)\Lam(\zeta_{\nu})=0$ is then
\begin{multline*}
 c^{\lam_2}_{\lam_1+\lam_2,\nu}
  ((\lam_1+\lam_2)(w_0)^2 - (\lam_2)(w_0)^2)
  g_{\lam_2}g_{\lam_1+\lam_2}
 +
  c^{\lam_2}_{\lam_1+3\lam_2,\nu}
  ((\lam_1+3\lam_2)(w_0)^2 - (\lam_2)(w_0)^2)
  g_{\lam_2}g_{\lam_1+3\lam_2}
 = 
 0.
\end{multline*}
Since $\lam_1(w_0)=0$ this is satisfied if $g_{\lam_1+3\lam_2}=0$.
We find that $i(\Xnh)\Lam_{(s,u)}$ vanishes when $(s,u)$ belongs
to the right invariant distribution
\begin{equation}\label{e:Dnew}
 \D_{\textup{new}}
 :=
 \ker(\eta^{\lam_1},\eta^{\lam_1+2\lam_2},\eta^{\lam_1+3\lam_2},\eta^{2\lam_1+3\lam_2})
 =
 \textup{span}\set{\zeta_{\lam_2},\zeta_{\lam_1+\lam_2}}.
\end{equation}
This is a rank two distribution with growth $2-3-5-6$
on a six dimensional configuration
space. 
Notice that
$[\zeta_{\lam_1},\D_{\textup{new}}]\subset\D_{\textup{new}}$, i.e.,
$\D_{\textup{new}}$ is invariant under the action of the connected
Lie group $H$ on $K$. Via the Langlands decomposition 
$H$ coincides with
$P_{w_0}\cap K\cong H \cong S^1$. 
Along $\D_{\textup{new}}$ the equations of motion are thus given by
the canonical equation
\[
 i(\Xnh)\Om^K = d\Hamc.
\]
Moreover, it is easy to see that $\Xnh$ is tangent to
$\D_{\textup{new}}$. (One could say that the constraint forces
vanish. However, this does of course not mean that the motion is
Hamiltonian since $\Xnh$ does not come from a Hamiltonian system on $TK$.) 
By invariance $\D_{\textup{new}}$ factors to a rank two distribution
$\D_{\textup{new}}/H$ of growth $2-3-5$ on
$K/H \cong \textup{SU}(2)\times\SO(3)/S^1 \cong S^3\times S^2$. 
Indeed, passing to the right
trivialization of $TK$ for a moment, $\D_{\textup{new}}/H$ can be
realized as 
\[
 K\times_{H}\textup{span}\set{Z_{\lam_2},Z_{\lam_1+\lam_2}}.
\]
Further, the restriction of the compressed Hamiltonian 
\[
 \Hamc|\D_{\textup{new}}
 =
 \by{1}{2}\vv<\mathbb{I}u,u>
 +
 \by{1}{2}\lam_2(w_0)^2(g_{\lam_2}^2 + g_{\lam_1+\lam_2}^2)
\]
is $K$-independent. 
E.g., 
$\zeta_{\lam_1}(g_{\lam_2}^2 + g_{\lam_1+\lam_2}^2)
 =
  -2c_{\lam_1,\lam_2}^{\lam_1+\lam_2}
   (g_{\lam_2}g_{\lam_1+\lam_2} - g_{\lam_1+\lam_2}g_{\lam_2}) 
 = 0$.
That is, $\Hamc|\D_{\textup{new}}$ is actually left
invariant.

Let us now follow \cite{S06} and define 
$\gu_i\subset\gu$ for $i\neq0$ to be the sum of all restricted root spaces
$\gu_{\lam}$ such that $\lam_{2}$ occurs with coefficient $i$ in the
decomposition of $\lam$ into simple roots $\lam_1,\lam_2$;
$\gu_0$ is defined to be the sum of $\ao$ and all restricted root spaces
$\gu_{\lam}$ such that $\lam_{2}$ occurs with coefficient $0$ in the
decomposition of $\lam$ into simple roots $\lam_1,\lam_2$.
Thus
\[ 
 \gu
 =
 \gu_{-3}\oplus\gu_{-2}\oplus\gu_{-1}\oplus\gu_0\oplus\gu_1\oplus\gu_2\oplus\gu_3
\]
which is the grading of $\gu$ with respect to the parabolic subalgebra
$\po_{w_0} = \textup{Lie}(P_{w_0}) = \oplus_{i=0,\ldots,3}\,\gu_i$. 
Choose an orthonormal basis $X_{\lam}$ of $\oplus_{\lam\in\Sigma}\,\gu_{\lam}$
consisting of root vectors.
Then the prescription $Z_{\lam}\mapsto X_{-\lam}$ and $e_{\lam}\mapsto
X_{\lam}$ for $\lam\in\Sigma_+$ induces isomorphisms
\[
 \ho^{\bot}\cong\gu_- 
  := \gu_{-3}\oplus\gu_{-2}\oplus\gu_{-1}
 \textup{ and }
 V\cong\gu_+
  := \gu_1\oplus\gu_2\oplus\gu_3 
   = \po_{w_0}/\gu_0.
\]
This corresponds effectively to the passage from the Cartan to the
Iwasawa decomposition.
Moreover, the isomorphism $\ho^{\bot}\cong\gu_-$ is equivariant with
respect to the $H$-action on $\ho^{\bot}$ and the $P_{w_0}$-action on
$\gu_-$. This follows from the Langlands decomposition of the
parabolic $P_{w_0}$.  
Associated to the grading there
is a $P_{w_0}$-invariant filtration 
\[
 \gu/\po_{w_0}\supset\gu^{-2}/\po_{w_0}\supset\gu^{-1}/\po_{w_0}
\]
of $\gu/\po_{w_0}$ where the filter components are $\gu^{i} =
\oplus_{j=i,\ldots,3}\,\gu_j$. With this notation and the isomorphism 
$\ho^{\bot}\cong\gu_-$
we obtain
\[
 \D_{\textup{new}}/H
 \cong 
 K\times_{H}\textup{span}\set{Z_{\lam_2},Z_{\lam_1+\lam_2}}
 \cong
 G\times_{P_{w_0}}\gu^{-1}/\po_{w_0}
 \subset
 G\times_{P_{w_0}}\gu/\po_{w_0}
 \cong
 T(S^3\times S^2).
\]
The growth of the distribution is of course reflected in the way in
which the filtration reacts to the Lie bracket: 
$[\gu^{-1}/\po_{w_0},\gu^{-1}/\po_{w_0}] = \gu^{-2}/\po_{w_0}$ and
$[\gu^{-1}/\po_{w_0},\gu^{-2}/\po_{w_0}] = \gu/\po_{w_0}$.
This distribution corresponds to the homogeneous model of Cartan
geometries of type $(G,P_{w_0})$. 

Bor and Montgomery~\cite{BM06} 
have explained that 
$G\times_{P_{w_0}}\gu^{-1}/\po_{w_0}
 \subset
 G\times_{P_{w_0}}\gu/\po_{w_0}$ 
can be identified with the no-twist no-slip distribution
when one passes over the two fold covering 
$S^3\times S^2 = K/H\to\SO(3)\times S^2$ 
and when the ratio of the radii of the two balls is $1/3$.
Along similar lines 
Sagerschnig~\cite{S06} has 
explained some of the Cartan geometric background and 
proved that
it is isomorphic to a certain `divisors of $0$ distribution', and 
Agrachev~\cite{A07} 
has shown that this `divisors of $0$
distribution' can be realized as the `rubber rolling distribution'
for ratio $1/3$.

\section{Questions}




Hamiltonization at non-zero momentum $\alpha\in\ho^*$ remains
open. Generalizing Theorem~\ref{thm:ham-at-0} to 
this setting is a problem for future work.
The difficulty here is that one has to take into account the extra
structure coming from the non-zero orbit $\orb=\Ad^*(H).\alpha$ in
\eqref{e:red-space}.



Integrability?
Very little is known about integrability of $n$-D Chaplygin
systems, and we have not touched at all the question of 
integrating the systems introduced in
Section~\ref{sec:chap-based-on-lie}.  
Jovanovic~\cite{Jov09} has just shown very recently that the
$n$-D Chaplygin ball
is integrable when the inertia tensor is of special type as in
\eqref{e:jov-i}. Of course, Chaplygin~\cite{Chap87} has explicitly
integrated the $3$-D problem.


\end{document}